\documentclass[11pt]{amsart}
\usepackage{graphicx}
\usepackage[space]{grffile}
\usepackage[T1]{fontenc}
\usepackage{color}
\usepackage[colorlinks,linkcolor=blue,citecolor=blue, pdfstartview=FitH]{hyperref}
\usepackage{backref}
\usepackage{amssymb}
\usepackage{epsf}
\usepackage{caption}
\usepackage{slashed}

\def\p{\partial}
\def\o{\overline}
\def\b{\bar}

\def\mb{\mathbb}
\def\mc{\mathcal}
\def\mr{\mathrm}

\def\mf{\mathfrak}
\def\n{\nabla}
\def\wt{\widetilde}

\def\s{\slashed}

\theoremstyle{plain}
\newtheorem{corollary}{Corollary}
\theoremstyle{plain}
\newtheorem{theorem}{Theorem}[section]
\theoremstyle{plain}
\newtheorem{definition}[theorem]{Definition}
\theoremstyle{plain}
\newtheorem{lemma}[theorem]{Lemma}
\theoremstyle{remark}
\newtheorem{remark}[theorem]{Remark}
\theoremstyle{remark}
\newtheorem{conjecture}[theorem]{Conjecture}
\theoremstyle{remark}

\theoremstyle{plain}
\newtheorem{proposition}[theorem]{Proposition}
\numberwithin{equation}{section}
\theoremstyle{plain}

\usepackage{fancyhdr}
\makeatletter
\@namedef{subjclassname@2020}{\textup{2020} Mathematics Subject Classification}
\makeatother

\begin{document}
%\begin{CJK}{UTF8}{gbsn}

\title{The mass of an asymptotically hyperbolic end and distance estimates}

\author{Xiaoxiang Chai}
\address{Xiaoxiang Chai: Korea Institute for Advanced Study, Seoul 02455, South Korea}
\email{xxchai@kias.re.kr}

\author{Xueyuan Wan}
\address{Xueyuan Wan: Mathematical Science Research Center, Chongqing University of Technology, Chongqing 400054, China.}
\email{xwan@cqut.edu.cn}

%\thanks{}
%\address{}
%\email{}
%\date{\today}

\begin{abstract}
Let $(M,g)$ be a complete connected $n$-dimensional Riemannian spin manifold without boundary such that the scalar curvature satisfies $R_g\geq -n(n-1)$ and $\mc{E}\subset M$ be an asymptotically hyperbolic end, we prove that the mass functional of the end $\mc{E}$ is timelike future-directed or zero. Moreover, it vanishes if and only if $(M,g)$ is isometric to the hyperbolic space. We also consider the mass of an asymptotically hyperbolic manifold with compact boundary, we prove the mass is timelike future-directed if the mean curvature of the boundary is bounded from below by a function defined using distance estimates. As an application, the mass is timelike future-directed if the mean curvature of the boundary is bounded from below by $-(n-1)$ or the scalar curvature satisfies $R_g\geq (-1+\kappa)n(n-1)$ for any positive constant $\kappa$ less than one. 

\end{abstract}


  \thanks{ Research of Xiaoxiang Chai is supported by KIAS Grants under the research code MG074402. Research of Xueyuan Wan is partially supported by the National Natural Science Foundation of China (Grant No. 12101093) and the Scientific Research Foundation of the Chongqing University of Technology.}
\maketitle

%\tableofcontents

\section{Introduction}

The Riemannian positive mass theorem of Schoen and Yau is an important research topic in the study of scalar curvature, mathematical relativity and geometric analysis.  Schoen and Yau \cite{SY1979, SY1981, S1989} proved this theorem for an asymptotically flat manifold of dimension $3\leq n\leq 7$. For higher dimensions, proofs were provided by Schoen and Yau \cite{SY2019} and Lohkamp \cite{L2016, L2017}.
Witten \cite{W1981} presented a spinorial proof, which works for spin manifolds of any dimension, see also \cite{B1986, PT1982}.

Schoen and Yau \cite{SY1988} conjectured that the positive mass theorem holds for an asymptotically flat manifold with arbitrary ends.
\begin{conjecture}[\cite{SY1988, LUY2020}]\label{PMC}
	Let $(M,g)$ be a complete Riemannian manifold of dimension $n\geq 3$ which has non-negative scalar curvature. Let $\mc{E}\subseteq M$ be a single asymptotically flat end in $M$. Then the ADM mass of $\mc{E}$ is non-negative.
\end{conjecture}
In \cite{LUY2021}, Lesourd, Unger and Yau proved Conjecture \ref{PMC} for $n\leq 7$ assuming that the chosen end is asymptotically Schwarzschild. Later the assumption that the chosen end is asymptotically Schwarzschild is removed in \cite{lee-density-2022-arxiv}. In the spin setting, 
Conjecture \ref{PMC} was proved by Bartnik and Chru\'sciel \cite[Theorem 11.2]{BC2003}
using Witten's method. Moreover, the ADM mass of the end $\mc{E}$ vanishes if and only if $(M,g)$ is isometric to the Euclidean space. In \cite{CZ2021}, Cecchini and Zeidler studied Conjecture \ref{PMC} in the spin setting from a quantitative point of view via a modification of Witten's method.

Motivated by Einstein's theory of negative cosmological constant, there is a lot of literature that define the mass for asymptotically hyperbolic manifolds and prove the positive mass theorem in this case, including \cite{M1989, AD1998, W2001, CH2003, AL2020}. In particular, for spin manifolds, there is the following version of the positive mass theorem. 
\begin{theorem}[{\cite[Theorem 4.1]{CH2003}}]\label{PMT: hyperbolic}
Let $(M,g)$ be a complete asymptotically hyperbolic spin manifold without boundary, such that the scalar curvature satisfies $R_g\geq -n(n-1)$, $n=\dim M$. Then the mass functional is timelike future-directed or zero. Moreover, it vanishes if and only if $(M,g)$ is isometric to the hyperbolic space. 
\end{theorem}
One can refer to Section \ref{section: AHM} for the definitions of asymptotically hyperbolic manifolds (ends) and the mass functional. 

Naturally, one may wonder whether the spin condition in Theorem \ref{PMT: hyperbolic} can be dropped. Without the spin condition, 
Andersson, Cai and  Galloway \cite{ACG2008} proved Theorem \ref{PMT: hyperbolic} for dimensions $3\leq n\leq 7$ , but under some restrictive conditions. Recently, Chru\'sciel and Delay \cite{CD2019} were able to reduce the hyperbolic positive mass theorem using a gluing argument to the positive mass theorem for asymptotically flat manifolds hence gave a proof in general dimensions. Inspired by Conjecture \ref{PMC} and the above results on asymptotically hyperbolic manifolds, one may pose the following conjecture on the positive mass theorem for asymptotically hyperbolic manifolds with arbitrary ends.
\begin{conjecture}\label{PMC: hperbolic}
	Let $(M,g)$ be a complete Riemannian manifold with scalar curvature $R_g\geq -n(n-1)$. Let $\mc{E}\subseteq M$ be an asymptotically hyperbolic end in $M$. Then the mass functional of the end $\mc{E}$ is timelike future-directed or zero.
	\end{conjecture}
Using Witten's spinorial argument and following the method of Cecchini and Zeidler \cite{CZ2021}, we can prove Conjecture \ref{PMC: hperbolic} in the spin setting. 
\begin{theorem}\label{main1}
Let $(M,g)$ be a complete connected $n$-dimensional Riemannian spin manifold without boundary such that $R_g\geq -n(n-1)$ and let $\mc{E}\subset M$ be an asymptotically hyperbolic end. Then the mass functional $\mf{m}_\mathcal{E}$ is timelike future-directed or zero. Moreover, it vanishes if and only if $(M,g)$ is isometric to the hyperbolic space. 
\end{theorem}
In fact, similar to \cite[Theorem A]{CZ2021}, we can give a quantification for the "positivity" of the mass functional $\mf{m}_\mathcal{E}$ of the asymptotically hyperbolic end $\mc{E}$. More precisely, let $\mc{U}_R(\mc{E})\subseteq M$ denote the open neighborhood of radius $R$ around $(\mc{E},g)$, if the mass is not timelike future-directed or zero, then we can obtain an explicit upper bound for $R$ such that $\o{\mc{U}_R(\mc{E})}$ is complete, spin and $R_g\geq -n(n-1)$, and it follows that $\o{\mc{U}_R(\mc{E})}\subsetneqq M$.
See Theorem \ref{quantification} for more details. 
  
We also consider the positive mass theorem for an asymptotically hyperbolic manifold with compact boundary. 
By deriving explicit distance estimates, we obtain
\begin{theorem}\label{thm:DE}
Let $(M,g)$ be an $n$-dimensional complete asymptotically hyperbolic	 spin manifold of  $R_g\geq -n(n-1)$ with compact boundary, $M_0\subseteq M_1\subseteq M$ be codimension zero submanifolds with boundary such that $M_0$ contains all asymptotically hyperbolic ends of $M$. Moreover, we assume that $R_g\geq (-1+\kappa)n(n-1)$ for some $0<\kappa<1$ on $M_1\backslash M_0$. We let $d=\mr{dist}_g(\p M_0,\p M_1)$ and $l=\mr{dist}_g(\p M_1,\p M)$ and define
\begin{align*}
\begin{split}
  \Psi(d,l):= \begin{cases}
 	\frac{2(n-1)}{(\frac{n}{\lambda(d)}+1)e^{-nl}-1},& \text{if }d<-t_0\text{ and }l<\frac{1}{n}\log(1+\frac{n}{\lambda(d)}),\\
 	\infty&\text{otherwise},
 \end{cases}
 \end{split}
\end{align*}
 where $t_0$ and $\lambda(d)$ are defined by \eqref{t0} and \eqref{lambda} respectively. 
 In this situation, if the mean curvature of $\p M$ satisfies 
 \begin{equation*}
  H_g+(n-1)>-\Psi(d,l)\text{ on } \p M,
\end{equation*}
then the mass functional $\mf{m}_\mathcal{E}$ of each end is timelike future-directed.
\end{theorem}
Note that $\Psi(d,l)$ is strictly positive  if $d<-t_0$ and $l<\frac{1}{n}\log(1+\frac{n}{\lambda(d)})$. As a corollary, we can prove 
\begin{corollary}[{\cite[Theorem 4.7]{CH2003}}]\label{cor1}
	Let $(M,g)$ be an $n$-dimensional complete asymptotically hyperbolic	 spin manifold of  $R_g\geq -n(n-1)$ with compact boundary. If $H_g+(n-1)\geq 0$ on all of $\partial M$, then the mass functional $\mf{m}_\mathcal{E}$ of each end is timelike future-directed.
	\end{corollary}
Note that our definition of mean curvature differs by a sign from \cite[Theorem 4.7]{CH2003}.
If  $R_g\geq (-1+\kappa)n(n-1)$ on all of $M$ for some $\kappa\in (0,1)$, then  the distance $d=\mr{dist}_g(\p M_0,\p M_1)$ can be large enough so that $d\geq -t_0$. As a result of Theorem \ref{thm:DE}, we obtain

\begin{corollary}\label{cor2}
	Let $(M,g)$ be an $n$-dimensional complete asymptotically hyperbolic	 spin manifold of  $R_g\geq (-1+\kappa)n(n-1)$ with compact boundary, where $\kappa\in (0,1)$ is a constant. Then the mass functional $\mf{m}_\mathcal{E}$ of each end is timelike future-directed.
\end{corollary}
\begin{remark}
	Note that Corollary \ref{cor2} also holds for an asymptotically hyperbolic end in a complete spin manifold without boundary. In fact, by Theorem \ref{main1}, the mass $\mf{m}_\mathcal{E}$ of an end $\mc{E}$ is timelike future-directed or zero. If it is zero, then $(M,g)$ is isometric to the hyperbolic space and it follows that $R_g=-n(n-1)$. This is contradictory to the condition $R_g\geq (-1+\kappa)n(n-1)$. Hence the mass $\mf{m}_\mathcal{E}$ is timelike future-directed.
\end{remark}

The article is organized as follows: In Section \ref{section: AHM}, we will recall the definitions of asymptotically hyperbolic manifolds (ends) and the mass functional. In Section \ref{CM}, we will recall the definitions of the Callias operators and derive the mass formulas using the Callias operators. In Section \ref{distance},  we will define a function $\Psi$ using explicit distance estimates and prove the positive mass theorem provided that the mean curvature of the boundary is bounded from below by $-\Psi$.  Theorem \ref{thm:DE}, Corollary \ref{cor1} and Corollary \ref{cor2} will be proved in this section.  In Section \ref{PMTAE}, we will consider the mass of an asymptotically end in a complete spin manifold without boundary, and we will prove Theorem \ref{main1}. In Section \ref{APP}, we will solve the boundary value problem associated with a Callias operator, the solutions are essentially used when estimating the mass. 

\section{Asymptotically hyperbolic manifolds and mass functional}\label{section: AHM}

In this section, we will recall the definitions of asymptotically hyperbolic manifolds (ends) and the mass functional. One can refer to \cite{CH2003, AL2022, HMR2015}.

\begin{definition}
	A Riemannian manifold $(M,g)$ is said to be {\it static} if there is a non-trivial solution $V$ of the following equations
\begin{equation*}
  	 \begin{cases}
 	\n^2_gV+\Lambda Vg-V\mr{Ric}_g=0,\\
 	\Delta_g V+\Lambda V=0
 \end{cases}
\end{equation*}
for some constant $\Lambda<0$. In this case, $V$ is called a static potential.
\end{definition}
The hyperboloid model for hyperbolic space is given by 
\begin{equation*}
  \mb{H}^n=\{x\in\mb{R}^{1,n}|\left\langle x,x\right\rangle_{1,n}=-1\}\subset \mb{R}^{1,n},
\end{equation*}
where $\mb{R}^{1,n}$ is the Minkowski space with the flat Lorentzian metric
\begin{equation*}
  \left\langle x,x\right\rangle_{1,n}=-x_0^2+x_1^2+\cdots+x_n^2,\quad x=(x_0,x_1,\cdots,x_n)\in\mb{R}^{1,n}.
\end{equation*}
The induced (Riemannian) metric on $\mb{H}^n$ is 
\begin{equation*}
  b=\frac{dr^2}{1+r^2}+r^2 g_{\mb{S}^{n-1}},
\end{equation*}
where $g_{\mb{S}^{n-1}}$ stands for the round metric on the unit sphere $\mb{S}^{n-1}$, $r=|x'|_\delta$, $x'=(0,x_1,\ldots,x_n)$ and
\begin{equation*}
  |x'|_\delta=\sqrt{x_1^2+\cdots+x_n^2}.
\end{equation*}
Then $(\mb{H}^n,b)$ is a complete static manifold with $\Lambda=-n$, whose space of static potentials is given by 
$$\mc{N}_b=\mr{Vect}\{V_{(0)},V_{(1)},\cdots, V_{(n)}\},$$
 where $V_{(0)}=x_0|_{\mb{H}^n}=\sqrt{r^2+1}$ and $V_{(i)}=x_i|_{\mb{H}^n}$, $1\leq i\leq n$, see \cite[(1.5) and (1.6)]{CH2003}. The space $\mc{N}_b$ is naturally endowed with a Lorentzian metric $\eta$ with the signature $(+,-,\cdots,-)$, and the metric $\eta$ satisfies
 \begin{equation*}
  \eta(V_{(0)},V_{(0)})=1,\quad \eta(V_{(i)},V_{(i)})=-1,
\end{equation*}
for $1\leq i\leq n$. 
 
Recall that one can parameterize $\mb{H}^n$ by polar coordinates $(r,\theta)$ where $r=|x'|_\delta$ and $\theta=(\theta_1,\ldots, \theta_n)\in\mb{S}^{n-1}$. Let $\epsilon_1,\ldots, \epsilon_{n-1}$ be an orthonormal frame for $g_{\mb{S}^{n-1}}$. Then 
$\{f_i\}_{i=1}^n$, with $f_a=r^{-1}\epsilon_a$, $a=1,\ldots,n-1$, and $f_n=\sqrt{r^2+1}\p_r$ is an orthonormal frame for $b$.
For any $r_0>0$, we denote 
\begin{equation*}
  \mb{H}^n_{r_0}:=\{x\in \mb{H}^n,r(x)\geq r_0\}.
\end{equation*}
Following \cite{CH2003}, the asymptotically hyperbolic ends and asymptotically hyperbolic manifolds can be defined as follows.
\begin{definition}[Asymptotically hyperbolic ends]
	Let $(M,g)$ be a smooth $n$-dimensional Riemannian manifold. We say an open subset $\mc{E}\subseteq M$ is an {\it asymptotically hyperbolic end} if there exists a diffeomorphism 
\begin{equation*}
  \Phi:\mathbb{H}^n_{r_0}\to \mathcal{E}
\end{equation*}
for some $r_0>0$, such that $\Phi^*(g|_{\mc{E}})$ and $b$ are uniformly equivalent on $\mb{H}^n_{r_0}$ and the metric components $g_{ij}:=(\Phi^*(g|_\mathcal{E}))(f_i,f_j)$ satisfy
\begin{equation}
|g_{ij}-\delta_{ij}|+|f_k (g_{ij})| = o(r^{-\tfrac{n}{2}}). \label{decay rate}
\end{equation}
Moreover, we assume that $r(R_g+n(n-1))$ belongs to $L^1(\mathcal{E},\mathbb{R})$ where $R_g$ denotes the scalar curvature of the Riemannian manifold $(M,g)$.
\end{definition}

\begin{definition}[Asymptotically hyperbolic manifolds]
A Riemannian manifold $(M,g)$ with compact boundary is said to be {\it asymptotically hyperbolic} if there exists a bounded subset $K\subset M$ whose complement $M\backslash K$ is a non-empty disjoint union of finitely many asymptotically hyperbolic ends $\mc{E}_1,\ldots,\mc{E}_N\subseteq M$.	
\end{definition}
\begin{figure}[ht]
\centering
\includegraphics[width=0.4\textwidth]{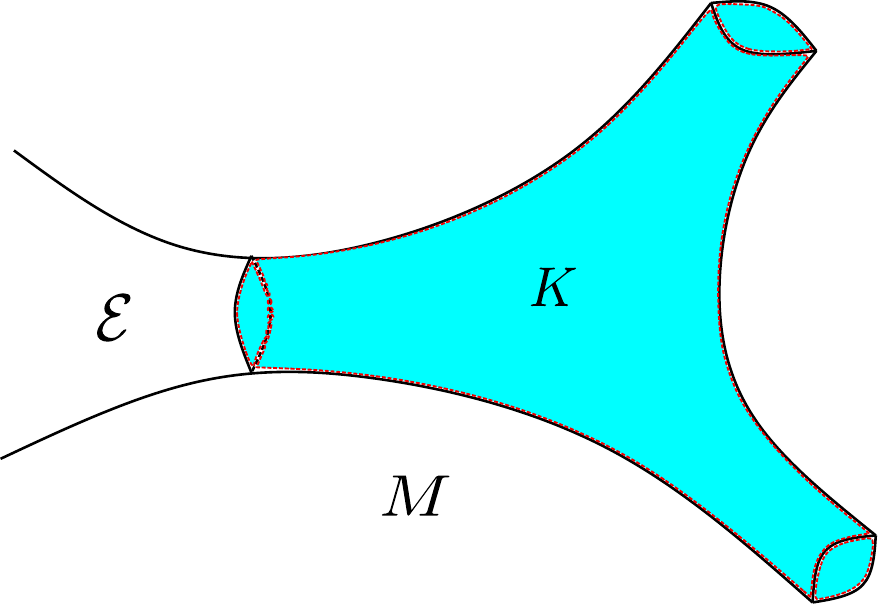}
\caption{Asymptotically hyperbolic manifolds.}
\end{figure}
Let $\mc{E}$ be an asymptotically hyperbolic end with a diffeomorphism $\Phi:\mb{H}^n_{r_0}\to \mc{E}$.
Denote $e:=\Phi^*g-b$ and define the following one form by
\begin{equation*}
  \mb{U}(V,e)=V(\mr{div}_be-d\mr{tr}_be)-i_{\n_b V} e+\mr{tr}_b edV,
\end{equation*}
where $i_X$ denotes the contraction by a vector field $X$. 
Set $S^{n-1}_{r}:=\{x\in \mb{H}^n_{r_0}:|x'|_\delta=r\}$, $r>r_0$. The mass functional $\mf{m}_\mathcal{E}$ on $\mc{N}_b$ is defined by
\begin{equation*}
  \mf{m}_\mathcal{E}(V)=\lim_{r\to+\infty}\int_{S^{n-1}_{r}}\mb{U}(V,e)(\nu_r) dS^{n-1}_{r}
\end{equation*}
for any $V\in \mc{N}_b$.  By \cite[Proposition 2.2]{CH2003}, $\mf{m}_\mathcal{E}(V)$ exists and is finite. Here $\nu_r$ denotes the unit infinity-pointing normal vector field to $S^{n-1}_{r}$.  

Following \cite{CH2003}, we recall the following definition of causal (timelike) future-directed linear functional.
\begin{definition}
We say the linear functional $\mf{m}_\mathcal{E}$ on $\mc{N}_b$ is {\it causal} (resp. {\it timelike}) {\it future-directed} if $(\mf{m}_\mathcal{E}(V_{(0)}),\cdots, \mf{m}_\mathcal{E}(V_{(n)}))\in \mathbb{R}^{1,n}$ in any $\eta$-orthonormal and future-oriented basis $(V_{(0)},\cdots, V_{(n)})$ satisfies
\begin{equation}
  (\mf{m}_\mathcal{E}(V_{(0)}))^2-\sum_{i=1}^n(\mf{m}_\mathcal{E}(V_{(i)}))^2 \label{invariant}
\end{equation}
is non-negative and $\mf{m}_\mathcal{E}(V_{(0)})\geq 0$ (resp. $\mf{m}_\mathcal{E}(V_{(0)})> 0$).
\end{definition}
\begin{remark}
According to \cite[Theorem 2.3]{CH2003}, \eqref{invariant} and its sign is actually independent of the chart $\Phi$. This also justifies our use of $\mf{m}_\mathcal{E}$. See also \cite{michel-geometric-2011}.
\end{remark}

\section{Callias operators and mass formulas}\label{CM}

In this section, we will recall the definition of the Callias operator and derive the formulas for the mass functional. 

\subsection{Callias operators}

In this subsection, we will recall the definition of the Callias operator, which can be found in \cite[Section 2.3]{CZ2021}.

Let $(M,g)$ be a complete asymptotically hyperbolic spin manifold with compact boundary. Let $\s{S}\to M$ denote the complex spinor bundle of $M$. Then 
$$S:=\s{S}\oplus\s{S}$$
 becomes an $\mb{Z}_2$-graded Dirac bundle with the induced connection 
$\n=\n_{\s{S}}\oplus\n_{\s{S}}$ and the Clifford multiplication
\begin{equation*}
  c(\xi)=\begin{pmatrix}
  0&c_{\s{S}}(\xi) \\
  c_{\s{S}}(\xi)&0 
\end{pmatrix},
\end{equation*}
 where $c_{\s{S}}$ and $\n_{\s{S}}$ are respectively the Clifford multiplication and connection on $\s{S}$. The Dirac operator on $S$ is given by
\begin{equation*}
  D=\begin{pmatrix}
  0&\s{D} \\
  \s{D}&0 
\end{pmatrix},
\end{equation*}
where $\s{D}:C^\infty(M,\s{S})\to C^\infty(M,\s{S})$ is the spinor Dirac operator on $(M,g)$. Set
\begin{equation*}
  \sigma:=\begin{pmatrix}
  0&-i \\
  i&0 
\end{pmatrix}.
\end{equation*}
 For a function $\psi\in C^\infty_c(M,\mb{R})$, the associated {\it Callias operator} is defined as
\begin{equation*}
  \mc{B}_\psi:=D+\psi\sigma.
\end{equation*}
For convenience, we also call $\mc{B}_\psi$ a Callias operator even if $\psi$ has no compact support. 
Let $\nu$ be the inward-pointing unit normal vector field to $\p M$. The {\it chirality operator} is defined as
\begin{equation}\label{chirality}
  \chi:=c(\nu^\flat)\sigma:S|_{\p M}\to S|_{\p M},
\end{equation}
where $\nu^\flat$ denotes the dual one form of $\nu$.
Denote by $C^\infty(M,S;\chi)$ the space of all smooth sections $u$ of $S$ such that $\chi(u|_{\p M})=u|_{\p M}$. We will use the analogous notation for other function spaces. 

\subsection{Spectral estimates} In this subsection, we will derive some spectral estimates of Callias operators for spin initial data sets. One can refer to \cite[Section 4.2]{CW2022} for the related calculations. 

Let $(M,g,k)$ be a spin initial data set where $k$ is a symmetric tensor of order two on a Riemannian manifold $(M,g)$.
Set
\begin{equation*}
  2\mu=R_g-|k|^2_g+(\mr{tr}_gk)^2,\quad  J_i=(\mr{div}_g k-d(\mr{tr}_gk))_i.
\end{equation*}
Recall that $\n=\n_{\s{S}}\oplus\n_{\s{S}}$ is the induced connection on the Dirac bundle $S=\s{S}\oplus\s{S}$, the associated Dirac operator is then  given by
\begin{equation*}
  D=c(e^i)\n_{e_i}.
\end{equation*}
One can define a new connection on $S$ by 
 \begin{equation*}
  \wt{\n}_{e_i}=\n_{e_i}-\tfrac{1}{2}k_{ij}c(e^j)\sigma,
\end{equation*}
which is indeed a connection  on $S$ since $\tfrac{1}{2}k_{ij}c(e^j)\sigma\in A^1(M,\mr{End}(S))$. Set
\begin{equation*}
  \wt{D}=c(e^i)\wt{\n}_{e_i}.
\end{equation*}
\begin{proposition}[{\cite[Proposition 4.6]{CW2022}}]\label{prop1}
For any smooth section $u\in C_c^\infty(M,S)$ with compact support, we have
\begin{itemize}
  \item[(i)] $\wt{\n}_{e_i}u=\n_{e_i}u+\frac{1}{2}k_{ji}\sigma c({e}^j)u$, $\wt{D}u=Du+ \tfrac{\mr{tr}_gk}{2}\sigma(u)$,
  \item[(ii)] $\n^*_{e_i}u=-\n_{e_i}u$, $\wt{\n}^*_{e_i}u=-\wt{\n}_{e_i}u-k_{ij}c({e}^j)\sigma(u)$, $D^*u=Du$, $\wt{D}^*u=\wt{D}u$;
  \item[(iii)] $D^2u=\n^*\n u+\mc{R}u$, $\wt{D}^2u=\wt{\n}^*\wt{\n}u+\wt{\mc{R}}u$, 
\end{itemize}
where $\mc{R}=\frac{R_g}{4}$ and $\wt{\mc{R}}=\frac{1}{2}(\mu-J_i c(e^i)\sigma)$.
\end{proposition}

In the following, the volume element will be omitted in the integral formula for convenience. The Green's formula for the Dirac operator $\wt{D}$ is then given by
$$
\int_{M}\langle\widetilde{D} u, v\rangle=\int_{M}\langle u, \widetilde{D} v\rangle+\int_{\partial M}\langle u, c(\nu^{b}) v\rangle,
$$
for any $u, v \in C_{c}^{\infty}(M, S)$. Here $\nu$ denotes the normal vector field pointing inward. We also have the following Green's formula for the connection $\wt{\n}$.
$$
\int_{M}\left\langle u, \widetilde{\nabla}^{*} \widetilde{\nabla} u\right\rangle=\int_{M}\langle\widetilde{\nabla} u, \widetilde{\nabla} u\rangle+\int_{\partial M}\left\langle u, \widetilde{\nabla}_{\nu} u\right\rangle.
$$
Let $M_r\subset M$ be a compact subset such that $M\backslash M_r$ is a non-empty disjoint union of finitely many asymptotically hyperbolic ends.
For any smooth section $u\in C^\infty(M,S)$, one has
\begin{align*}
\begin{split}
  \int_{M_r} |\wt{D}u|^2&=\int_{M_r} \left\langle u,\wt{D}^2u\right\rangle+\int_{\p {M_r}} \left\langle u,c(\nu^\flat)\wt{D}u\right\rangle\\
  &=\int_{M_r} \left\langle u,\wt{\n}^*\wt{\n}u\right\rangle+\int_{M_r} \left\langle u,\wt{\mc{R}}u\right\rangle+\int_{\p M_r} \left\langle u,c(\nu^\flat)\wt{D}u\right\rangle\\
  &=\int_{M_r} |\wt{\n}u|^2+\int_{M_r} \left\langle u,\wt{\mc{R}}u\right\rangle+\int_{\p {M_r}}\left\langle u,c(\nu^\flat)\wt{D}u+\wt{\n}_{\nu}u\right\rangle.
 \end{split}
\end{align*}
The boundary Dirac operator is defined as follows
\begin{align*}
\begin{split}
  \mc{A}:=\sum_{i=1}^{n-1}c^\p(e^i)\n^\p_{e_i},
 \end{split}
\end{align*}
where $e_n=-\nu$ and $c^\p(e^i)=c(e^i)c(\nu^\flat)$ and $\n^{\p}_{e_i}=\n_{e_i}+\frac{1}{2}c^\p(\n_{e_i}\nu^{\flat})$.
%, see e.g. \cite[Section 2]{cecchini-scalar-2021-arxiv}.
 Denote by $h$ the second fundamental form on $\p M$, then $\mr{tr}_{\p M}h=\sum_{i=1}^{n-1}\left\langle e_i,\n_{e_i}(-\nu)\right\rangle$. Thus
\begin{align*}
\begin{split}
  \mc{A}&=\frac{1}{2}\mr{tr}_{\p M}h-c(\nu^\flat) D-\n_{\nu}\\
  &=\frac{1}{2}\mr{tr}_{\p M}h-c(\nu^b)(\wt{D}-\frac{\mr{tr}_gk}{2}\sigma)-(\wt{\n}_{\nu}+\frac{1}{2}k_{j\nu}c({e}^j)\sigma)\\
  &=\frac{1}{2}\mr{tr}_{\p M}h+\frac{\mr{tr}_gk}{2}c(\nu^\flat)\sigma-\frac{1}{2}k_{\nu\nu}c(\nu^\flat)\sigma-\frac{1}{2}k_{a\nu}c( e^a)\sigma-(c(\nu^\flat)\wt{D}+\wt{\n}_{\nu})\\
  &=\frac{1}{2}\mr{tr}_{\p M}h+\frac{\mr{tr}_{\p M}k}{2}c(\nu^\flat)\sigma-\frac{1}{2}k_{a\nu}c( e^a)\sigma-(c(\nu^\flat)\wt{D}+\wt{\n}_{\nu}),
 \end{split}
\end{align*}
which follows that
\begin{align*}
\begin{split}
   \int_{M_r} |\wt{D}u|^2&=\int_{M_r} |\wt{\n}u|^2+\int_{M_r} \left\langle u,\wt{\mc{R}}u\right\rangle\\
   &\quad +\int_{\p M}\left\langle u,(\frac{1}{2}\mr{tr}_{\p M}h+\frac{\mr{tr}_{\p M}k}{2}c(\nu^\flat)\sigma-\frac{1}{2}k_{a\nu}c( e^a)\sigma-\mc{A})u\right\rangle\\
   &\quad +\int_{\p M_r\backslash \p M}\left\langle u,c(\nu^\flat)\wt{D}u+\wt{\n}_{\nu}u\right\rangle.
 \end{split}
\end{align*}
The Penrose operator of $\wt{\n}$ is defined as 
\begin{align*}
\begin{split}
  \wt{\mc{P}}_\xi u=\wt{\n}_\xi u+\frac{1}{n}c(\xi^\flat) \wt{D}u.
 \end{split}
\end{align*}
The Friedrich inequality is then 
\begin{align*}
\begin{split}
|\wt{\n}u|^2-\frac{1}{n}|\wt{D}u|^2 = |\wt{\mc{P}}u|^2\geq 0,
 \end{split}
\end{align*}
which follows that
\begin{align*}
\begin{split}
 &  \int_{M_r} |\wt{D}u|^2=\frac{n}{n-1}\int_{M_r} |\wt{\mc{P}}u|^2+\frac{n}{n-1}\int_{M_r} \left\langle u,\wt{\mc{R}}u\right\rangle\\
   &+\frac{n}{n-1}\int_{\p M}\left\langle u,(\frac{1}{2}\mr{tr}_{\p M}h+\frac{\mr{tr}_{\p M}k}{2}c(\nu^\flat)\sigma-\frac{1}{2}k_{a\nu}c( e^a)\sigma-\mc{A})u\right\rangle\\
   & \quad +\frac{n}{n-1}\int_{\p M_r\backslash \p M}\left\langle u,c(\nu^\flat)\wt{D}u+\wt{\n}_{\nu}u\right\rangle.
 \end{split}
\end{align*}
For any $\psi\in C^\infty_c(M,\mb{R})$, the Callias operator is defined as
\begin{equation*}
{  \mc{B}_{\psi}=\wt{D}+\psi \sigma=D+(\frac{1}{2}\mr{tr}_gk+\psi)\sigma}.
\end{equation*}
For any $u\in C^\infty(M,S)$, one has
\begin{multline*}
  \int_{M_r} | \mc{B}_{\psi}u|^2=\int_{M_r} (|\wt{D}u|^2+|\psi|^2|u|^2)\\
  +\int_{M_r}\left\langle u,(c(\mathrm{d}\psi)\sigma+\psi (\mr{tr}_gk) )u\right\rangle+\int_{\p M_r}\left\langle u,\psi c(\nu^b)\sigma(u)\right\rangle.
\end{multline*}
By taking $r$ sufficiently large, then $\psi=0$ on $\p M_r\backslash \p M$. Using  $\left\langle u, c(\mathrm{d}\psi)\sigma u\right\rangle\leq |\mathrm{d}\psi||u|^2$, so 
\begin{align*}
\begin{split}
   \int_{M_r} | \mc{B}_{\psi}u|^2\geq &\frac{n}{n-1}\int_{M_r} |\wt{\mc{P}}u|^2+\frac{n}{n-1}\int_{M_r} \left\langle u,\wt{\mc{R}}u\right\rangle\\
   & +\frac{n}{n-1}\int_{\p M}\left\langle u,(\frac{1}{2}\mr{tr}_{\p M}h+\frac{\mr{tr}_{\p M}k}{2}c(\nu^\flat)\sigma-\frac{1}{2}k_{a\nu}c( e^a)\sigma-\mc{A})u\right\rangle\\
   &+\int_{M_r} (|\psi|^2-|\mathrm{d}\psi|+\psi \mr{tr}_gk)|u|^2+\int_{\p M}\left\langle u,\psi c(\nu^b)\sigma(u)\right\rangle\\
   & +\frac{n}{n-1}\int_{\p M_r\backslash \p M}\left\langle u,c(\nu^\flat)\wt{D}u+\wt{\n}_{\nu}u\right\rangle.
 \end{split}
\end{align*}
Now we consider $u\in C^\infty(M,S;\chi)$ that satisfies the boundary condition $\chi(u|_{\p M})=u|_{\p M}$, then 
\begin{align*}
  \int_{M_r} | \mc{B}_{\psi}u|^2 &\geq \frac{n}{n-1}\int_{M_r} |\wt{\mc{P}}u|^2+\frac{n}{n-1}\int_{M_r} \left\langle u,\wt{\mc{R}}u\right\rangle\\
  &\quad + \int_{M_r} (|\psi|^2-|\mathrm{d}\psi|+\psi \mr{tr}_gk)|u|^2\\
  &\quad +\frac{n}{n-1}\int_{\p M}\left\langle u,(\frac{1}{2}\mr{tr}_{\p M}h+\frac{\mr{tr}_{\p M}k}{2}+\frac{n-1}{n}\psi)u\right\rangle\\
  &\quad +\frac{n}{n-1}\int_{\p M_r\backslash \p M}\left\langle u,c(\nu^\flat)\wt{D}u+\wt{\n}_{\nu}u\right\rangle.
\end{align*}
Since $\left\langle u,\wt{\mc{R}}u\right\rangle\geq \frac{1}{2}(\mu-|J|)|u|^2$, denote by $H_g:=\mr{tr}_{\p M}h$ the mean curvature of $\partial M$, so
\begin{align}\label{spectral1}
\begin{split}
&\quad \int_{M_r} | \mc{B}_{\psi}u|^2\\
&\geq  \frac{n}{n-1}  \int_{M_r} \left(\frac{n-1}{n}(|\psi|^2-|\mathrm{d}\psi|+\psi \mr{tr}_gk)+\frac{1}{2}(\mu-|J|)\right)|u|^2\\
  &\quad +\frac{n}{n-1}\int_{\p M}\left\langle u,(\frac{1}{2}H_g+\frac{\mr{tr}_{\p M}k}{2}+\frac{n-1}{n}\psi)u\right\rangle\\
  &\quad+\frac{n}{n-1}\int_{M_r} |\wt{\mc{P}}u|^2+ \frac{n}{n-1}\int_{\p M_r\backslash \p M}\left\langle u,c(\nu^\flat)\wt{D}u+\wt{\n}_{\nu}u\right\rangle.
  \end{split}
\end{align}
Similarly, if we do not use the Friedrich inequality, then
\begin{align}\label{spectral2}
\begin{split}
&\quad \int_{M_r} | \mc{B}_{\psi}u|^2\\
&\geq    \int_{M_r} \left((|\psi|^2-|\mathrm{d}\psi|+\psi \mr{tr}_gk)+\frac{1}{2}(\mu-|J|)\right)|u|^2\\
  &\quad +\int_{\p M}\left\langle u,(\frac{1}{2}H_g+\frac{\mr{tr}_{\p M}k}{2}+\psi)u\right\rangle\\
  &\quad+\int_{M_r} |\wt{\n}u|^2+ \int_{\p M_r\backslash \p M}\left\langle u,c(\nu^\flat)\wt{D}u+\wt{\n}_{\nu}u\right\rangle
  \end{split}
\end{align}
for any $u\in C^\infty(M,S;\chi)$.

\subsection{Mass formulas}

In this subsection, we will derive some formulas for the mass functional using \eqref{spectral1} and \eqref{spectral2}. 

In the hyperboloid model, the hyperbolic space $(\mathbb{H}^n,b)$ is the upper half of the 2-sheeted hyperboloid in the Minkowski spacetime. The second fundamental form of $(\mathbb{H}^n,b)$ is $\pm g$ with the sign depending on the choice of the orientation of the unit normal. So we consider $k=g$. Then
\begin{equation*}
  \wt{\n}_{e_i}=\n_{e_i}-\frac{1}{2}g_{ij}c(e^j)\sigma
\end{equation*}
and the Dirac operator is 
\begin{align*}
\wt{D}=D+\frac{n}{2}\sigma:C^\infty(M,S)\to C^\infty(M,S).
\end{align*}
In this case, $J=0$ and
\begin{equation*}
  \wt{R}=\frac{1}{2}\mu=\frac{R_g+n(n-1)}{4}.
\end{equation*}
Let $M_r$ be a compact subset such that $K\subset M_r\subset M$, and $M\backslash M_r$ is a union of finitely many asymptotically hyperbolic ends.
 By \eqref{spectral1}, for any $u\in C^\infty(M,S;\chi)$, one has
\begin{align*}
\begin{split}
  &\quad -\frac{n}{n-1}\int_{\p M_r\backslash \p M}	\left\langle u,c(\nu^\flat)\wt{D}u+\wt{\n}_{\nu}u\right\rangle
+\int_{M_r} | \mc{B}_{\psi}u|^2 \\
&\geq  \frac{n}{n-1}  \int_{M_r} \left(\frac{n-1}{n}(|\psi|^2-|\mathrm{d}\psi|+\psi n)+\frac{1}{4}(R_g+n(n-1))\right)|u|^2\\
  &\quad +\frac{n}{n-1}\int_{\p M}\left\langle u,(\frac{1}{2}H_g+\frac{n-1}{2}+\frac{n-1}{n}\psi)u\right\rangle +\frac{n}{n-1}\int_{M_r}|\wt{\mc{P}}u|^2.
   \end{split}
\end{align*}

Set
\begin{equation}\label{bartheta}
  \b{\theta}_\psi:=\frac{n}{n-1}\frac{R_g+n(n-1)}{4}+\psi^2-|d\psi|+n\psi
\end{equation}
and 
\begin{equation}\label{bareta}
   \b{\eta}_\psi:=\frac{n}{2(n-1)}H_g+\frac{n}{2}+\psi|_{\p M}.
\end{equation}

For any $u\in C^\infty(M,S;\chi)$ such that $\mc{B}_\psi u\in L^2(M,S)$,
by taking $r\to\infty$, one has
\begin{align}\label{eqn0}
\begin{split}
& \quad  -\frac{n}{n-1}\lim_{r\to\infty}\int_{\p M_r\backslash \p M}	\left\langle u,c(\nu^\flat)\wt{D}u+\wt{\n}_{\nu}u\right\rangle+\|\mc{B}_\psi u\|^2_{L^2(M,S)}\\
&\geq \frac{n}{n-1}\|\tilde{\mc{P}}u\|^2_{L^2(M,S)}+\int_M \b{\theta}_\psi |u|^2 dV+\int_{\p M}\b{\eta}_\psi|u|^2 dS.
 \end{split}
\end{align}
Similarly, by setting
\begin{equation}\label{theta}
  \theta_\psi=\frac{R_g+n(n-1)}{4}+\psi^2-|d\psi|+n\psi
\end{equation}
and 
\begin{equation}\label{eta}
\eta_\psi=\frac{H_g}{2}+\frac{n-1}{2}+\psi|_{\p M},
\end{equation}
then
\begin{align}
\begin{split}
 &\quad  -\lim_{r\to\infty}\int_{\p M_r\backslash \p M}	\left\langle u,c(\nu^\flat)\wt{D}u+\wt{\n}_{\nu}u\right\rangle+\|\mc{B}_\psi u\|^2_{L^2(M,S)}\\
&\geq \|\widetilde{\n}u\|^2_{L^2(M,S)}+\int_M {\theta}_\psi |u|^2 dV+\int_{\p M}{\eta}_\psi|u|^2 dS.
 \end{split}
\end{align}
From the definitions of $\wt{D}$ and $\wt{\n}$, one has
\begin{align*}
\begin{split}
  c(\nu^\flat)\wt{D}+\wt{\n}_{\nu} =c(\nu^\flat)D+\n_\nu+\frac{n-1}{2}c(\nu^\flat)\sigma.
 \end{split}
\end{align*}
We assume that $u=(u_1,u_2)\in C^\infty(M,S)= C^\infty(M,\s{S}\oplus\s{S})$, then 
\begin{align*}
& \quad ( c(\nu^\flat)\wt{D}+\wt{\n}_{\nu})(u)\\
&=\left((c(\nu^\flat)\s{D}+\n_\nu+\frac{n-1}{2}ic(\nu^\flat))u_1,(c(\nu^\flat)\s{D}+\n_\nu-\frac{n-1}{2}ic(\nu^\flat))u_2
\right).
\end{align*}

\begin{definition}
The Killing connections and the Killing Dirac operators are respectively defined as 
\begin{equation*}
  \n^{\pm}_X=\n_X\pm\frac{i}{2}c(X^\flat)
\end{equation*}
 and
 \begin{equation*}
  \s{D}^\pm:=c\circ \n^\pm=\s{D}\mp \frac{ni}{2},
\end{equation*}
for any vector field $X$. $u$ is called an imaginary Killing spinor if $\n^\pm u=0$, see e.g. \cite{CH2003, AL2022}.
\end{definition}
By the above definition, one has
\begin{equation*}
  c(\nu^\flat)\s{D}^\pm+\n^\pm_\nu=c(\nu^\flat)\s{D}+\n_{\nu}\mp \frac{n-1}{2}ic(\nu^\flat).
\end{equation*}
Hence 
\begin{align*}
\begin{split}
 &\quad \left\langle u, c(\nu^\flat)\wt{D}u+\wt{\n}_{\nu}u\right\rangle\\
 &=\left\langle u_1,c(\nu^\flat)\s{D}^-u_1+\n^-_\nu u_1\right\rangle+\left\langle u_2,c(\nu^\flat)\s{D}^+u_2+\n^+_\nu u_2\right\rangle.
 \end{split}
\end{align*}
Similarly, one has
\begin{align}\label{BD}
\begin{split}
  |\mc{B}_\psi u|^2&=|(\s{D}^-+\psi i)u_1|^2+|(\s{D}^+-\psi i)u_2|^2\\
  &=|\s{D}^-_\psi u_1|^2+|\s{D}^+_\psi u_2|^2,
  \end{split}
\end{align}
where $\s{D}^\pm_\psi:=\s{D}^\pm\mp\psi i$,
and
\begin{equation*}
  |\wt{\n}u|^2=|\n^- u_1|^2+|\n^+ u_2|^2.
\end{equation*}
In a word, we obtain
\begin{align}\label{eqn1}
\begin{split}
 &\quad  -\lim_{r\to\infty}\int_{\p M_r\backslash \p M}\left(\left\langle u_1,c(\nu^\flat)\s{D}^-u_1+\n^-_\nu u_1\right\rangle+\left\langle u_2,c(\nu^\flat)\s{D}^+u_2+\n^+_\nu u_2\right\rangle\right)\\
&\geq-\|\mc{B}_\psi u\|^2_{L^2(M,S)}+ \|\widetilde{\n}u\|^2_{L^2(M,S)}+\int_M {\theta}_\psi |u|^2 dV+\int_{\p M}{\eta}_\psi|u|^2 dS.
 \end{split}
\end{align}
\begin{remark}
	By taking $u_1=0$ or $u_2=0$, the above identity is equivalent to 
\begin{align*}
\begin{split}
 & -\lim_{r\to\infty}\int_{\p M_r\backslash \p M}\left\langle u_0,c(\nu^\flat)\s{D}^\pm u_0+\n^\pm_\nu u_0\right\rangle+\|\s{D}^\pm_\psi u_0\|^2_{L^2(M,\s{S})}\\
 &\geq \|\n^\pm u_0\|^2_{L^2(M,\s{S})}+\int_M {\theta}_\psi |u_0|^2 dV+\int_{\p M}{\eta}_\psi|u_0|^2 dS.
 \end{split}
\end{align*}
for any $u_0\in C^\infty(M,\s{S})$ satisfying $ic(\nu^\flat)u_0=\pm u_0$ and $\s{D}^\pm_\psi u_0\in L^2(M,\s{S})$. 
\end{remark}

Following  \cite[Remark 3.17]{AL2022}, the hyperbolic space $\mb{H}^n$ can be described in terms of Poincar\'e ball model 
\begin{equation*}
  \mb{B}^n=\{y\in\mb{R}^n; |y|_\delta<1\},\quad \widehat{b}=\Omega(y)^{-2}\delta,
\end{equation*}
where $\Omega(y)=(1-|y|^2_\delta)/2$. There is an identification between the spinor bundles $\mb{S}\mb{B}^n_\delta$ and $\mb{S}\mb{B}^n_{\widehat{b}}$. If $u_0\in C^\infty(\mb{B}^n,\mb{S}\mb{B}^n_\delta)$ is a $\n^\delta$-parallel spinor, then 
\begin{equation*}
  \Xi_{u_0,\pm}(y):=\Omega(y)^{-1/2}\o{(I\mp ic^\delta(y))u_0}
\end{equation*}
satisfying $\n^{\hat{b},\pm}\Xi_{u_0,\pm}=0$. Furthermore,
\begin{equation*}
  V^\pm_{u_0}:=|\Xi_{u_0,\pm}(y)|^2_b\in \mc{N}_b,
\end{equation*}
see \cite[Proposition 4.1]{AL2022}. Let $\mc{E}\subset M$ be an asymptotically hyperbolic end in $M$. We set
\begin{equation*}
  \Psi_{u_0,\pm}:=\chi_{\mc{E}}\Xi_{u_0,\pm}+\phi_{u_0,\pm},
\end{equation*}
where $\chi_{\mc{E}}$ is a cut-off function that vanishes outside the end $\mc{E}$ and is equal to $1$ in $(M\backslash M_r)\cap \mc{E}$ for sufficiently large $r$, $\phi_{u_0,\pm}$ are chosen  by the following lemma.
\begin{lemma}
There exists a $v=(\phi_{u_0,-},\phi_{u_0,+})\in L^2_1(M,S;\chi)$  such that 
\begin{equation}\label{harmonic spinor}
  \s{D}^\pm_\psi\Psi_{u_0,\pm}=\s{D}^\pm_\psi(\chi_{\mc{E}}\Xi_{u_0,\pm}+\phi_{u_0,\pm})=0.
\end{equation}
\end{lemma}
\begin{proof}
	For any $\n^\delta$-parallel spinor $u_0\in C^\infty(\mb{B}^n,\mb{S}\mb{B}^n_\delta)$, denote 
\begin{equation*}\label{eqn2}
  u_1=\Psi_{u_0,-},\quad u_2=\Psi_{u_0,+}, \quad u=(u_1,u_2).
\end{equation*}
From \eqref{BD}, then \eqref{harmonic spinor} is equivalent to 
$
  \mc{B}_\psi u=0.
$
Now we consider the following boundary value problem
\begin{equation*}
   \begin{cases}
\mc{B}_\psi u 	 =0,\\
\chi(u|_{\p M}) =u|_{\p M}.
 \end{cases}
\end{equation*}
Note that 
\begin{align*}
\begin{split}
  \mc{B}_\psi u&=(\s{D}^+_\psi \Psi_{u_0,+},\s{D}^-_\psi \Psi_{u_0,-})\\
  &=(\s{D}^+_\psi (\chi_\mc{E}\Xi_{u_0,+}+\phi_{u_0,+}),\s{D}^-_\psi (\chi_\mc{E}\Xi_{u_0,-}+\phi_{u_0,-}))\\
  &=\mc{B}_\psi(\chi_\mc{E}\Xi_{u_0,-},\chi_{\mc{E}}\Xi_{u_0,+})+\mc{B}_\psi(\phi_{u_0,-},\phi_{u_0,+}).
 \end{split}
\end{align*}
The same proof \cite[(4.17)]{CH2003} shows that $\s{D}^\pm_\psi(\chi_\mc{E}\Xi_{u_0,\pm})\in L^2(M,\s{S})$, so 
$$\mc{B}_\psi(\chi_\mc{E}\Xi_{u_0,-},\chi_\mc{E}\Xi_{u_0,+})\in L^2(M,S).$$ 
Using Proposition \ref{BVP}, there exists such a $v=(\phi_{u_0,-},\phi_{u_0,+})\in L^2_1(M,S;\chi)$ satisfying $\mc{B}_\psi v=-\mc{B}_\psi(\chi_\mc{E}\Xi_{u_0,-},\chi_\mc{E}\Xi_{u_0,+})$, so $\mc{B}_\psi u=0$. The proof is complete.
\end{proof}

Since $\psi$ has compact support, by \eqref{harmonic spinor}, so
\begin{equation*}
  \s{D}^\pm\Psi_{u_0,\pm}=\s{D}^\pm_\psi\Psi_{u_0,\pm}=0
\end{equation*}
on $M\backslash M_r$ for sufficiently large $r$. Hence
\begin{align*}
\begin{split}
 &\quad  -\lim_{r\to\infty}\int_{\p M_r\backslash \p M}\left\langle \Psi_{u_0,\pm},c(\nu^\flat)\s{D}^\pm \Psi_{u_0,\pm}+\n^\pm_\nu  \Psi_{u_0,\pm}\right\rangle\\
 &=  -\lim_{r\to\infty}\int_{(\p M_r\backslash \p M)\cap\mc{E}}\left\langle \Psi_{u_0,\pm},c(\nu^\flat)\s{D}^\pm \Psi_{u_0,\pm}+\n^\pm_\nu \Psi_{u_0,\pm}\right\rangle\\
  &=  -\lim_{r\to\infty}\int_{(\p M_r\backslash \p M)\cap\mc{E}}\left\langle \Psi_{u_0,\pm},\n^\pm_\nu \Psi_{u_0,\pm}\right\rangle\\
 &=\frac{1}{4}\mf{m}_\mathcal{E}(V_{u_0}^\pm),
 \end{split}
\end{align*}
where  the last equality follows from  \cite[(4.22)]{CH2003}.

By \eqref{eqn1}, one has
\begin{align}\label{eqn3}
\begin{split}
&\quad  \frac{1}{4}\mf{m}_\mathcal{E}(V_{u_0}^-)+\frac{1}{4}\mf{m}_\mathcal{E}(V_{v_0}^+)\\
  &\geq\|\widetilde{\n}u\|^2_{L^2(M,S)}+\int_M {\theta}_\psi |u|^2 dV+\int_{\p M}{\eta}_\psi|u|^2 dS,
 \end{split}
\end{align}
where $u=(u_1,u_2)=(\Psi_{u_0,-},\Psi_{v_0,+})$.
\begin{remark}
The above inequality is also equivalent to 
\begin{align*}
\begin{split}
  \frac{1}{4}\mf{m}_\mathcal{E}(V_{u_0}^\pm)&\geq-\|\n^\pm \Psi_{u_0,\pm}\|^2_{L^2(M,\s{S})}+\int_M {\theta}_\psi |\Psi_{u_0,\pm}|^2 dV+\int_{\p M}{\eta}_\psi|\Psi_{u_0,\pm}|^2 dS
 \end{split}
\end{align*}
by considering $u_0=0$ or $v_0=0$.
\end{remark}

Similarly, one has by \eqref{spectral2}
\begin{align}
\begin{split}
& \quad  \frac{n}{(n-1)}( \frac{1}{4}\mf{m}_\mathcal{E}(V_{u_0}^-)+\frac{1}{4}\mf{m}_\mathcal{E}(V_{v_0}^+))\\
&\geq \frac{n}{n-1}\|\tilde{\mc{P}}u\|^2_{L^2(M,S)}+\int_M \b{\theta}_\psi |u|^2 dV+\int_{\p M}\b{\eta}_\psi|u|^2 dS.
 \end{split}
\end{align}

If $\bar{\theta}_\psi\geq 0$ and $\bar{\eta}_\psi\geq 0$, then 
\begin{equation*}
  \mf{m}_\mathcal{E}(V^-_{u_0})\geq 0\text{ and }\mf{m}_\mathcal{E}(V^+_{v_0})\geq 0
\end{equation*}
for any $\n^\delta$-parallel spinors $u_0, v_0\in C^\infty(\mb{B}^n,\mb{S}\mb{B}^n_\delta)$. From \cite[Proposition 4.1]{AL2022}, one has 
\begin{equation*}
  V^\pm_{u_0}\in \mc{C}_b^\uparrow:=\{V\in\mc{N}_b|\eta(V,V)=0,\eta(V,V_{(0)})>0\}\subset \mc{N}_b
\end{equation*}
if $V^\pm_{u_0}\not\equiv 0$,  and moreover any $V\in\mc{C}_b^\uparrow$ can be written as $V=V^\pm_{u_0}$ for some such $u_0$. Hence the mass functional $\mf{m}_\mathcal{E}$ is non-negative on the future-pointing isotropic cone $\mc{C}^\uparrow_b\subset\mc{N}_b$. Standard considerations in Lorentzian geometry yield that the linear functional $\mf{m}_\mathcal{E}$ is timelike future-directed or zero.

 In a word, we obtain
\begin{theorem}\label{thm1}
	Let $(M,g)$ be a complete asymptotically hyperbolic spin manifold with compact boundary and let $\psi\in C^\infty_c(M,\mb{R})$. Suppose that $\b{\theta}_\psi\geq 0$ and $\b{\eta}_\psi\geq 0$. Then the mass functional $\mf{m}_\mathcal{E}$ of each end is timelike  future-directed or zero on each end. If $\b{\theta}_\psi(x)>0$ or $\b{\eta}_\psi(x)>0$ for some point $x\in M$ or $x\in \p M$ respectively, then  the mass functional $\mf{m}_\mathcal{E}$ of each end is timelike future-directed.
\end{theorem}

\section{The neck of AH manifolds}\label{distance}

Following the proof of \cite[Theorem C]{CZ2021}, we will give a similar result in the hyperbolic setting. More precisely, we will define a function $\Psi$ using explicit distance estimates, the mass can be proved to be timelike future-directed if the mean curvature of the boundary is bounded from below by $-\Psi$. 

For any $\kappa\in(0,1)$, we consider the following ordinary differential equation
\begin{equation*}
  \kappa\frac{n^2}{4}+y(t)^2-y'(t)+ny(t)=0,
\end{equation*}
which has a solution 
\begin{equation*}
  y(t)=-\frac{n}{2}\left(1+\sqrt{1-\kappa}\coth(\frac{n}{2}\sqrt{1-\kappa}t)\right)
\end{equation*}
with $y'(t)>0$. Now we consider the function $y=y(t)$ defined on the interval  $(-\infty,0)$. Denote 
\begin{equation}\label{t0}
  t_0:=\frac{2}{n\sqrt{1-\kappa}}\mr{arccoth}(-\frac{1}{\sqrt{1-\kappa}})<0.
\end{equation}
Note that $t_0\to-\infty$ as $\kappa\to 0$. By the definition of $t_0$, $y(t_0)=0$. Let $\delta>0$ be a constant such that
\begin{equation*}
  t_0+\delta=\frac{2}{n\sqrt{1-\kappa}}\mr{arccoth}(-\frac{1}{\sqrt{1-\kappa}})+\delta<0,
\end{equation*}
which is equivalent to
\begin{equation*}
 \sqrt{1-\kappa} \coth(\frac{n}{2}\sqrt{1-\kappa}\delta)>1.
\end{equation*}
Set
\begin{align}\label{lambda}
\begin{split}
  \lambda(\delta)&:=-\frac{n}{2}\left(1+\sqrt{1-\kappa}\coth(\frac{n}{2}\sqrt{1-\kappa}(\delta+t_0))\right)\\
  &=\frac{n}{2}\cdot\frac{2-\kappa}{\sqrt{1-\kappa}\coth(\frac{n}{2}\sqrt{1-\kappa}\delta)-1}.
 \end{split}
\end{align}

Similar to \cite[Lemmas 3.2, 3.3]{CZ2021}, we have the following two lemmas.
\begin{lemma}\label{lemma:fun1}
  Let $N$ be a compact manifold with boundary such that $\p N=\p_-N\sqcup \p_+N$ where $\p_\pm N$ are non-empty unions of components.
Let $\kappa,\delta$ be positive constants such that $0<\kappa<1$ and $\delta+t_0<0$.
 Then there exists a smooth function 
 $$p:N\to [0,\lambda(\delta)]$$
  such that $p=0$ in a neighborhood of $\p_-N$, $p=\lambda(\delta)$ in a neighborhood of $\p_+N$ and $\kappa\frac{n^2}{4}+p^2-|dp|+np\geq 0$.

\end{lemma}
\begin{proof}
As $\mr{dist}_g(\p_-N,\p_+N)>\delta$, there exists a smooth $1$-Lipschitz function $x:N\to\mb{R}$ such that $x|_{\p_-N}=-\epsilon+t_0$ and $x|_{\p_+N}=\delta+t_0+\epsilon$, for some $\epsilon>0$. Consider $\wt{p}(t)=-\frac{n}{2}\left(1+\sqrt{1-\kappa}\coth(\frac{n}{2}\sqrt{1-\kappa}t)\right)$, then 
	\begin{equation*}
  \kappa\frac{n^2}{4}+\wt{p}^2-\wt{p}'+n\wt{p}=0,
\end{equation*}
and $\wt{p}(t_0)=0$ and $\wt{p}(\delta+t_0)=\lambda(\delta)$. By a slightly modifying $\wt{p}$, we obtain a smooth function ${p}_\epsilon:[-\epsilon+t_0,\delta+t_0+\epsilon]\to[0,\lambda(\delta)]$ such that ${p}_\epsilon=0$ in a neighborhood of $-\epsilon+t_0$, $p_\epsilon=\lambda(\delta)$ in a neighborhood of $\delta+t_0+\epsilon$, and $ \kappa\frac{n^2}{4}+{p}_\epsilon^2-{p}'_\epsilon+n{p}_\epsilon\geq 0$. Finally, by setting $p:=p_\epsilon\circ x$, we obtain a smooth function on $N$.
\end{proof}

\begin{lemma}\label{lemma:fun2}
	Let $N$ be a compact manifold with boundary such that $\p N=\p_-N\sqcup \p_+N$ where $\p_\pm N$ are non-empty unions of components. Let $\lambda>0$ be arbitrary and suppose that $\mr{dist}_g(\p_-N,\p_+ N)>\delta$ for some $\delta\in (0,\frac{1}{n}\log(1+\frac{n}{\lambda}))$. Then there exists a smooth function $h:N\to [\lambda,\infty)$ such that $h=\lambda$ in a neighborhood of $\p_-N$, $h=\frac{n}{(\frac{n}{\lambda}+1)e^{-n\delta}-1}$ in a neighborhood of $\p_+ N$ and $h^2-|dh|+nh\geq 0$.
\end{lemma}
\begin{proof}
	As $\mr{dist}_g(\p_-N,\p_+N)>\delta$, there exists a smooth $1$-Lipschitz function $x:N\to\mb{R}$ such that $x|_{\p_-N}=-\epsilon+t_0$ and $x|_{\p_+N}=\delta+t_0+\epsilon$, for some $\epsilon>0$. Now we take
	\begin{equation*}
  \tilde{h}(t)=\frac{n}{(\frac{n}{\lambda}+1)e^{-nt}-1},\quad t\in [0,\frac{1}{n}\log(1+\frac{n}{\lambda})),
\end{equation*}
then 
\begin{align*}
\begin{split}
  (\tilde{h}\circ x)^2-|d(\tilde{h}\circ x)|+n(\tilde{h}\circ x)=\frac{n^2(\frac{n}{\lambda}+1)e^{-nt}}{((\frac{n}{\lambda}+1)e^{-nt}-1)^2}(1-|dx|)\geq 0.
 \end{split}
\end{align*}
By slightly modify to $\tilde{h}$, we complete the proof. 
\end{proof}

\begin{theorem}\label{thm:main2}
Let $(M,g)$ be an $n$-dimensional complete asymptotically hyperbolic	 spin manifold of  $R_g\geq -n(n-1)$ with compact boundary. Let $M_0\subseteq M_1\subseteq M$ be codimension zero submanifolds with boundary such that $M_0$ contains all asymptotically hyperbolic ends of $M$. Moreover, we assume that $R_g\geq (-1+\kappa)n(n-1)$ for some $0<\kappa<1$ on $M_1\backslash M_0$. We let $d=\mr{dist}_g(\p M_0,\p M_1)$ and $l=\mr{dist}_g(\p M_1,\p M)$ and define
\begin{align*}
\begin{split}
  \Psi(d,l):= \begin{cases}
 	\frac{2(n-1)}{(\frac{n}{\lambda(d)}+1)e^{-nl}-1},& \text{if }d<-t_0\text{ and }l<\frac{1}{n}\log(1+\frac{n}{\lambda(d)})\\
 	\infty&\text{otherwise}
 \end{cases}
 \end{split}
\end{align*}
 where $\lambda(d)$ is defined by \eqref{lambda} and
\begin{equation*}
  \frac{n}{\lambda(d)}+1=\frac{2(\sqrt{1-\kappa}\coth(\frac{n}{2}\sqrt{1-\kappa}d))-\kappa}{2-\kappa}.
\end{equation*}
 In this situation, if the mean curvature of $\p M$ satisfies 
 \begin{equation*}
  H_g+(n-1)>-\Psi(d,l)\text{ on } \p M,
\end{equation*}
then $\mf{m}_\mathcal{E}$ is timelike future-directed on each end of $M$.
\end{theorem}
\begin{proof}
\begin{figure}[ht]
\centering
\includegraphics[width=0.5\textwidth]{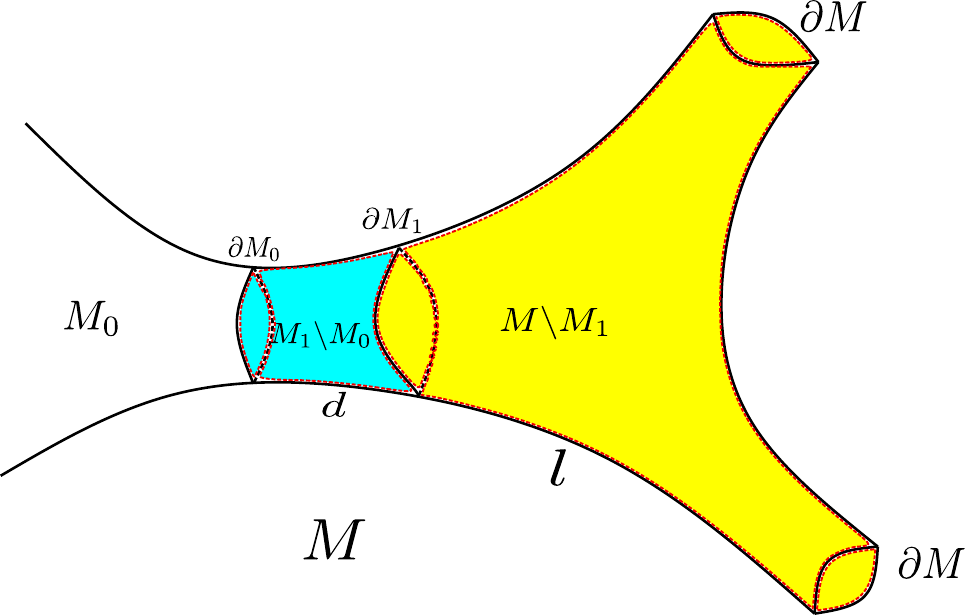}
\caption{$M_0\subseteq M_1\subseteq M$.}
\end{figure}
By continuity, we can find $d'<d$ and $l'<l$ such that $d'<-t_0$, $l'<\frac{1}{\lambda(d')}$ and $H_g+(n-1)>-\Psi(d',l')>-\infty$. Since 
\begin{equation*}
  \mr{dist}_g(\p M_0,\p M_1)=d>d',
\end{equation*}
by Lemma \ref{lemma:fun1},
there exists a smooth function $p:M_1\backslash M_0\to[0,\lambda(d')]$ such that $p=0$ in a neighborhood of $\p M_0$, $p=\lambda(d')$ in a neighborhood of $\p M_1$ and 
\begin{equation*}
  \kappa\frac{n^2}{4}+p^2-|dp|+np\geq 0.
\end{equation*}
Since $\mr{dist}_g(\p M_1,\p M)=l>l'$, by Lemma \ref{lemma:fun2}, there exists a smooth function $h:M\backslash M_1\to [\lambda(d'),\infty)$ such that $h=\lambda(d')$ in a neighborhood of $\p M_1$, 
$$h|_{\p M}=\frac{n}{(\frac{n}{\lambda(d')}+1)e^{-nl'}-1}=\frac{n}{2(n-1)}\Psi(d',l')$$
and 
$$h^2-|dh|+nh\geq 0.$$
Finally, let $\psi\in C^\infty_c(M,\mb{R})$ be defined by setting $\psi|_{M_0}=0$, $\psi|_{M_1\backslash M_0}=p$ and $\psi|_{M\backslash M_1}=h$. By this construction, one has $\b{\theta}_\psi\geq 0$. Since $h|_{\p M}=\frac{1}{2}\Psi(d',l')$, so $\bar{\eta}_\psi>0$. From Theorem \ref{thm1}, the proof is complete.
\end{proof}

\section{The positive mass theorem with arbitrary ends}\label{PMTAE}

In this section, we will consider the mass of an asymptotically end in a complete spin manifold without boundary.

 \begin{theorem}\label{quantification}
 Let $(\mc{E},g)$ be an $n$-dimensional asymptotically hyperbolic end such that $\mf{m}_\mathcal{E}(V_{u_0}^-)+\mf{m}_\mathcal{E}(V_{v_0}^+)<0$ for some $\n^\delta$-parallel $u_0,v_0\in C^\infty(\mb{B}^n,\mb{S}\mb{B}^n_\delta)$.	Then there exists a constant $R(\mc{E},g)$ such that the following holds: If $(M,g)$ is an $n$-dimensional Riemannian manifold without boundary that contains $(\mc{E},g)$ as an open subset and $\mc{U}=\mc{U}_R(\mc{E})\subseteq M$ denotes the open neighborhood of radius $R$ around $\mc{E}$ in $M$, then at least one of the following conditions must be violated:
 \begin{itemize}
  \item[(a)] $\o{\mc{U}}$ is (metrically) complete,
  \item[(b)] $\inf_{x\in\mc{U}}R_g(x)\geq -n(n-1)$,
  \item[(c)] $\mc{U}$ is spin.
\end{itemize}
 \end{theorem}
\begin{figure}[ht]
\centering
\includegraphics[width=0.6\textwidth]{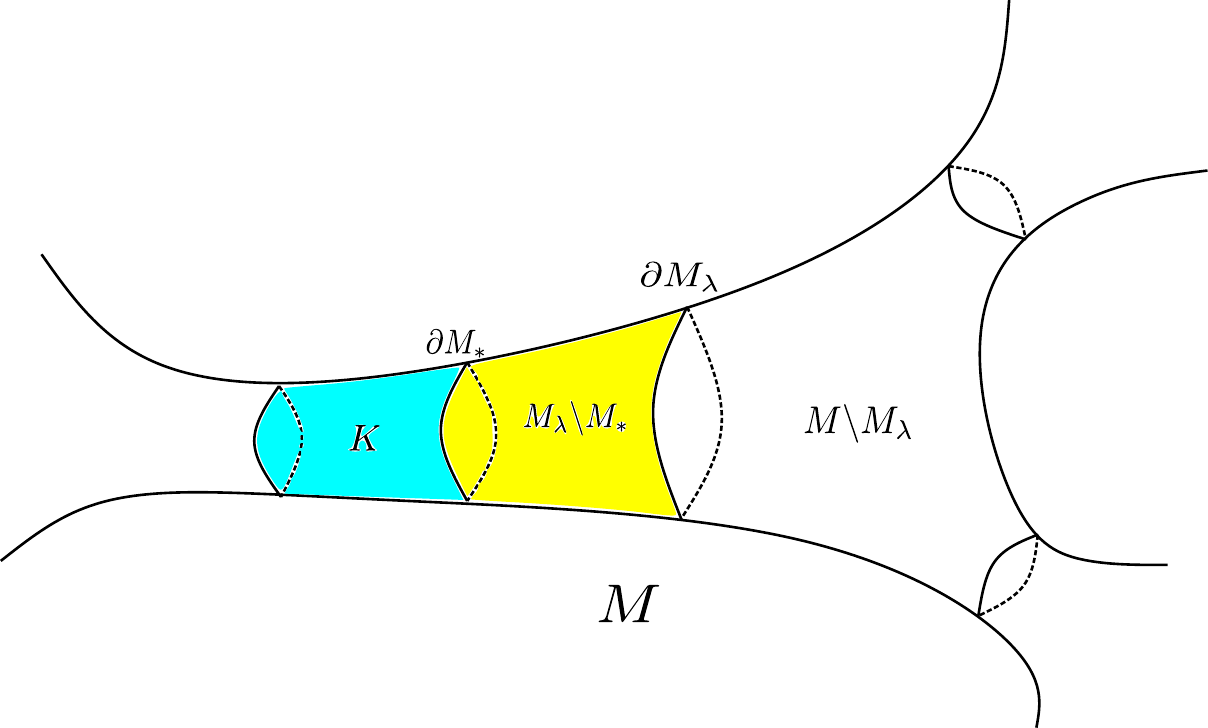}
\caption{The definitions of $K,M_*,M_\lambda$.}
\end{figure}
\begin{proof}
	Firstly, we assume that $\mc{E}$ satisfies (a)-(c), or we can take $R(\mc{E},g)=0$.
	Let $M_{*}\subseteq \mc{E}$ be a subset which is diffeomorphic to $\mb{H}^n_{r_0}$ for some $r_0>0$. We furthermore assume that some neighborhood of $M_*$ in $M$ still satisfies (a)-(c). Let $K\subset M_*$ be the closure of a collar neighborhood of $\p M_*$ in $M_*$. Let $f:M_*\to[0,1]$ be a smooth function such that $f=0$ on $M_*\backslash K$, and $f=1$ in a small neighborhood of $\p M_*$. 
Recall that $u=(u_1,u_2)$ with
$$ u_1=\Psi_{u_0,-}=\chi_{\mc{E}}\Xi_{u_0,-}+\phi_{u_0,-},\quad u_2=\Psi_{v_0,+}=\chi_{\mc{E}}\Xi_{v_0,+}+\phi_{v_0,+},$$
where $\chi_\mc{E}$ is a cut-off function with $\mr{supp}(\chi_\mc{E})\subseteq M_*\backslash K$.

Let $M_{\lambda}$ ($\supset M_*$) be a subset of $M$ satisfying (a)-(c) and 
$$\mr{dist}_g(\p M_*,\p M_{\lambda})\geq \frac{1}{n}\log(1+\frac{n}{\lambda}).$$
	Denote by $\lambda_0$ the infimum of such $\lambda$. We fix such a submanifold $M_{\lambda}$ for each $\lambda>\lambda_0$. 
	
	For each $\lambda>\lambda_0$, by Lemma \ref{lemma:fun2}, we can choose a smooth function $h_\lambda:M_{\lambda}\backslash M_*\to [\lambda,+\infty)$ such that $h_\lambda=\lambda$ in a small neighborhood of $\p M_*$ which satisfies
	\begin{equation*}
 h^2-|dh|+nh\geq 0
\end{equation*}
in all of $M_{\lambda}\backslash M_*$, and $h_\lambda|_{\p M_\lambda}+\frac{1}{2}H_{\lambda}+\frac{1}{2}(n-1)>0$. Let $\psi_\lambda:M_\lambda\to [0,\infty)$ be the smooth function defined by 
\begin{align*}
\begin{split}
  \psi_\lambda(x)= \begin{cases}
 	\lambda f(x)&x\in M_*, \\
 	h_\lambda(x)& x\in M_\lambda\backslash M_*.
 \end{cases}
 \end{split}
\end{align*}
Let $\mc{B}_{\psi_\lambda}$ be the associated Callias operator and $\eta_\lambda,\theta_\lambda$ are defined by \eqref{theta}, \eqref{eta} respectively. By construction, we obtain
\begin{itemize}
  \item[(i)] $\theta_\lambda\geq 0$ in $M_\lambda\backslash K$;
  \item[(ii)] $\theta_\lambda\geq -\lambda |df|$ in $K$;
  \item[(iii)] $\eta_\lambda\geq 0$.
\end{itemize}
In particular, the negative part of $\theta_\lambda$ is supported in $K\subset M_*$ and satisfies $(\theta_\lambda)_-\leq \lambda |df|$.
	By Proposition \ref{pro:WPI}, denote by $w_*$ the weight of the Poincar\'e inequality of $\wt{\n}$ over $M_*$, so $\operatorname{ess} \inf_{K} w_*=c_0>0$. 

	 If $\lambda_1:=\frac{c_0}{2\|df\|_\infty}>\lambda_0$, for any $\lambda\in(\lambda_0,\lambda_1)$, let $w_\lambda$ be the associated weight on $M_\lambda$, and set
	 \begin{equation*}
  w(x)= \begin{cases}
 w_*(x)	& x\in M_* \\
 	w_\lambda(x) & x\in M_\lambda\backslash M_*.
 \end{cases}
\end{equation*}
Then $w(x)$ is also a weight of $\wt{\n}$ over $M_\lambda$.
Moreover, on $K\subset M_*$, one has
 $$w-(\theta_\lambda)_->c_0-\lambda \|df\|_\infty=\frac{c_0}{2}>0,\quad \text{ a.e. }.$$
 Since $\mr{supp}(\chi_\mc{E})\cap \mr{supp}(\psi_\lambda)=\emptyset$, we have 
 \begin{align*}
 \mc{B}_{\psi_\lambda}(\chi_\mc{E}\Xi_{u_0,-},\chi_{\mc{E}}\Xi_{v_0,+})&=(\s{D}_\psi^+(\chi_\mc{E}\Xi_{v_0,+}),\s{D}^-_\psi(\chi_{\mc{E}}\Xi_{u_0,-}))\\
 &\in L^2(M_*,S;\chi)\subset L^2(M_\lambda,S;\chi).
 \end{align*}
	Using Proposition \ref{BVP1}, there exists a $v_\lambda\in L^2_1(M_\lambda,S;\chi)$ such that
	\begin{equation*}
  \mc{B}_{\psi_\lambda}(v_\lambda)=-\mc{B}_{\psi_\lambda}(\chi_\mc{E}\Xi_{u_0,-},\chi_{\mc{E}}\Xi_{v_0,+}).
\end{equation*}
Define $u_\lambda:=v_\lambda+(\chi_\mc{E}\Xi_{u_0,-},\chi_{\mc{E}}\Xi_{v_0,+})$. Then $\mc{B}_{\psi_\lambda}(u_\lambda)=0$ and $u_\lambda=v_\lambda$ on $(M_\lambda\backslash M_*)\cup K$.  By \eqref{eqn3}, one has
\begin{align}\label{eqn6.1}
\begin{split}
0&>  \frac{1}{4}\mf{m}_\mathcal{E}(V_{u_0}^-)+\frac{1}{4}\mf{m}_\mathcal{E}(V_{v_0}^+)\\
  &\geq \|\widetilde{\n}u_\lambda\|^2_{L^2(M_\lambda,S)}+\int_{M_\lambda} {\theta}_\psi |u_\lambda|^2 dV+\int_{\p M_\lambda}{\eta}_\psi|u_\lambda|^2 dS\\
  &\geq\|\widetilde{\n}u_\lambda\|^2_{L^2(K,S)}-\int_{K}(\theta_\lambda)_-|u_\lambda|^2 dV\\
  &=\|\wt{\n}v_\lambda\|^2_{L^2(K,S)}-\int_{K}(\theta_\lambda)_-|v_\lambda|^2 dV\\
  &\geq \int_{K}(w_*-(\theta_\lambda)_-)|v_\lambda|^2 dV\geq 0,
 \end{split}
\end{align}
which is a contradiction. Hence 
$$\lambda_0\geq \lambda_1=\frac{c_0}{2\|df\|_\infty}>0,$$
which means that $\lambda_0$ admits a strictly positive lower bounded which depends on the end $(\mc{E},g)$. The proof is complete.
\end{proof}
\begin{theorem}
Let $(M,g)$ be a complete connected $n$-dimensional Riemannian spin manifold without boundary such that $R_g\geq -n(n-1)$ and let $\mc{E}\subset M$ be an asymptotically hyperbolic end. Then the mass functional $\mf{m}_\mathcal{E}$ is timelike future-directed or zero. Moreover, it vanishes if and only if $(M,g)$ is isometric to the hyperbolic space $(\mb{H}^n,b)$. 
\end{theorem}
\begin{proof}
Since $(M,g)$ is complete, spin and $R_g\geq -n(n-1)$, so $\lambda_0=0$. Hence $\lambda_1>\lambda_0$. For any $\lambda\in (\lambda_0,\lambda_1)$,  one has by \eqref{eqn6.1}
	\begin{align*}
\begin{split}
   \frac{1}{4}\mf{m}_\mathcal{E}(V_{u_0}^-)+\frac{1}{4}\mf{m}_\mathcal{E}(V_{v_0}^+)\geq \int_{K}(w_0-(\theta_\lambda)_-)|v_\lambda|^2 dV\geq 0
 \end{split}
\end{align*}
for any $\n^\delta$-parallel $u_0,v_0\in C^\infty(\mb{B}^n,\mb{S}\mb{B}^n_\delta)$. Hence $\mf{m}_\mathcal{E}$ is timelike future-directed or zero.

If $\mf{m}_\mathcal{E}$ is zero, then there exist $\n^\delta$-parallel $u_0,v_0\in C^\infty(\mb{B}^n,\mb{S}\mb{B}^n_\delta)$ such that $\mf{m}_\mathcal{E}(V_{u_0}^-)=\mf{m}_\mathcal{E}(V_{v_0}^+)=0$ and $(u_0,v_0)\neq 0$. Hence 
\begin{align}
\begin{split}
0&=	\frac{1}{4}\mf{m}_\mathcal{E}(V_{u_0}^-)+\frac{1}{4}\mf{m}_\mathcal{E}(V_{v_0}^+)\\
&\geq \frac{1}{2}\|\wt{\n}u_\lambda\|^2_{L^2(M_\lambda,S)}+ \frac{1}{2}\|\wt{\n}v_\lambda\|^2_{L^2(K,S)}-\int_{K}(\theta_\lambda)_-|v_\lambda|^2 dV                                        \\
&\geq \frac{1}{2}\|\wt{\n}u_\lambda\|^2_{L^2(M_\lambda,S)}+\frac{1}{2}\int_{K}(w_0-2(\theta_\lambda)_-)|v_\lambda|^2 dV.
\end{split}
\end{align}
For any $\lambda\in (0,\lambda_1)$, one has 
\begin{equation*}
\|\wt{\n}u_\lambda\|^2_{L^2(M_*,S)}\leq  \|\wt{\n}u_\lambda\|^2_{L^2(M_\lambda,S)}=0.
\end{equation*}
Hence, for any $\lambda,\lambda'\in (0,\lambda_1)$, one has
\begin{align*}
\begin{split}
  \wt{\n}(v_\lambda-v_{\lambda'})=\wt{\n}(u_\lambda-u_{\lambda'})=0
 \end{split}
\end{align*}
on $M_*$, which follows that $v_\lambda=v_{\lambda'}$ on $M_*$ by the weighted Poincar\'e inequality \eqref{WPI}. 

On the other hand, one has 
\begin{equation*}
 \| \wt{\n}v_\lambda\|^2_{L^2(M_\lambda\backslash M_*,S)}= \| \wt{\n}u_\lambda\|^2_{L^2(M_\lambda\backslash M_*,S)}\leq  \| \wt{\n}u_\lambda\|^2_{L^2(M_\lambda,S)}=0,
\end{equation*}
which follows that $v_\lambda\equiv 0$ on $M_\lambda\backslash M_*$ by the weighted Poincar\'e inequality \eqref{WPI}. For any $\lambda\in (0,\lambda_1)$, then the following function
\begin{equation*}
  v(x)= \begin{cases}
 	v_\lambda(x)& x\in M_*, \\
0& 	x\in M\backslash M_*.
 \end{cases}
\end{equation*}
is well-defined (independent of $\lambda$), and set
\begin{equation*}
  u=v+(\chi_\mc{E}\Xi_{u_0,-},\chi_{\mc{E}}\Xi_{v_0,+}).
\end{equation*}
Then $u(x)=0$ for $x\in M\backslash M_*$, and for any $x\in M_*$
\begin{equation*}
 ( \wt{\n}u)(x)=(\wt{\n}(v_\lambda+(\chi_\mc{E}\Xi_{u_0,-},\chi_{\mc{E}}\Xi_{v_0,+})))(x)=(\wt{\n}u_\lambda)(x)=0.
\end{equation*}
Since $u=(u_1,u_2)$ is non-trivial, it follows that $u_1$ or $u_2$ is a non-trivial imaginary Killing spinor on $M$.
From the work of H. Baum \cite{B1989} or \cite[Theorem 3.1 and Lemma 4.11]{AD1998}, $(M,g)$ is isometric to $(\mb{H}^n,b)$.

\end{proof}

\section{Appendix: Boundary value problems}\label{APP}      

In this section, we will solve a boundary value problem associated with a Callias operator.          

Recall that the Clifford bundle over $M$ is $S=\s{S}\oplus\s{S}$, the involution, Clifford multiplication and Dirac operator are 
\begin{equation*}
  \sigma=\begin{pmatrix}
  0&-i \\
  i&0 
\end{pmatrix},\quad c_{S}=\begin{pmatrix}
  0&c_{\s{S}} \\
  c_{\s{S}}&0 
\end{pmatrix},\quad   D=\begin{pmatrix}
  0&\s{D} \\
  \s{D}&0 
\end{pmatrix}.
\end{equation*}
The Callias operator is defined as 
\begin{equation*}
\mc{B}_\psi=\wt{D}+\psi\sigma=D+(\frac{n}{2}+\psi)\sigma.
\end{equation*}
\begin{definition}[{\cite[Definition 8.2]{BB2012}}]\label{coercive}
We say that $\mc{B}_\psi$ is coercive at infinity if there is a compact subset $K \subset M$ and a constant $C$ such that
$$
\|u\|_{L^{2}(M,S)} \leq C\|\mc{B}_\psi u\|_{L^{2}(M,S)},
$$
for all smooth sections $u$ with compact support in $M \backslash K$.	
\end{definition}

Following the argument in \cite[Proposition 4.7]{AL2020}, one has
\begin{proposition}\label{BVP}
If $(M,g)$ is a complete asymptotically hyperbolic spin manifold with compact boundary and let $\psi\in C^\infty_c(M,\mb{R})$. Suppose that $\b{\theta}_\psi\geq 0$ and $\bar{\eta}_\psi\geq 0$. Then for any $u\in C^\infty(M,S)$ such that $\mc{B}_\psi u\in L^2(M,S)$, there exists a unique $v\in L^2_1(M,S;\chi)$ solving the boundary value problem
\begin{align}\label{eqn:BVP}
\begin{split}
   \begin{cases}
 	\mc{B}_\psi v=-\mc{B}_\psi u&\text{in } M, \\
 	\chi(v|_{\p M})=v|_{\p M} &\text{on } \p M,
 \end{cases}
 \end{split}
\end{align}
where $\chi$ is the chirality operator \eqref{chirality}.
\end{proposition}
\begin{proof}
	For any $u,v\in L^2_1({M,S})$, one has
	\begin{equation*}
  \int_M\left(\left\langle \mc{B}_\psi u,v\right\rangle-\left\langle u,\mc{B}_\psi v \right\rangle\right)=\int_{\p M}\left\langle u,c(\nu^\flat)v\right\rangle.
\end{equation*}
 Hence
\begin{align*}
\begin{split}
 \mc{B}_\psi:\mr{dom} \mc{B}_\psi\subset L^2_1(M,S;\chi)\to L^2(M,S)
 \end{split}
\end{align*}
is a self-adjoint operator with the chirality boundary condition. For any $u\in \mr{dom} \mc{B}_\psi$, by \eqref{eqn0} and noting that the first term in \eqref{eqn0} vanishes since $u\in L^2_1(M,S;\chi)$, one has
\begin{align*}
\begin{split}
  \|\mc{B}_\psi u\|^2_{L^2(M,S)}\geq \frac{n}{n-1}\|\widetilde{\mc{P}}u\|^2_{L^2(M,S)}+\int_M \b{\theta}_\psi |u|^2 dV+\int_{\p M}\b{\eta}_\psi|u|^2 dS.
 \end{split}
\end{align*}
If $u\in \ker\mc{B}_\psi$, then $\widetilde{\mc{P}}u=0$, i.e. for any vector field $\xi$, one has
$$\wt{\n}_\xi u+\frac{1}{n}c(\xi^\flat) \wt{D}u=\wt{\n}_\xi u-\frac{\psi}{n}c(\xi^\flat)\sigma u=0.$$
Since $\psi$ has compact support, so $\wt{\n}u=0$ on $M\backslash K$ for some compact subset $K$. Equivalently, $\n^- u_1=0$ and $\n^+u_2=0$ on $M\backslash K$. If $u_1$ or $u_2$ is not trivial, then they vary exponentially along geodesic segment, therefore can not lie in $\mr{dom}\mc{B}_\psi$. Hence $u=(u_1,u_2)\equiv 0$ on $M\backslash K$, which follows that $u\equiv 0$ on $M$ since $\wt{\n}_\xi u-\frac{\psi}{n}c(\xi^\flat)\sigma u=0$.
 Hence $\ker \mc{B}_\psi=\{0\}$.
 
Let $K$ be a compact subset of $M$ such that $\psi=0$ on $M\backslash K$. For any smooth section $u\in C^\infty(M,S)$ with compact support in $M\backslash K$, one has
 \begin{align*}
\begin{split}
  \|\mc{B}_\psi u\|^2_{L^2(M,S)}&=\int_M |\mc{B}_\psi u|^2=\int_M|(D+\frac{n}{2}\sigma)u|^2\\
  &=\int_M |Du|^2+\frac{n^2}{4}\int_M |u|^2+\frac{n}{2}\int_M(\left\langle D u,\sigma u\right\rangle+\left\langle \sigma u,Du\right\rangle)\\
  &=\|Du\|^2_{L^2(M,S)}+\frac{n^2}{4}\|u\|^2_{L^2(M,S)}+\frac{n}{2}\int_{\p M}\left\langle u,c(\nu^\flat)\sigma u\right\rangle\\
  &\geq \frac{n^2}{4}\|u\|^2_{L^2(M,S)},
 \end{split}
\end{align*}
where the last inequality follows from $u=0$ on $\p M$. Hence $\mc{B}_\psi$ is coercive at infinity by Definition \ref{coercive}.

 Using \cite[Proposition 4.19]{GN2014} and taking $\rho=0$ (which can be chosen since $\ker \mc{B}_\psi=0$), then one can solve a unique $v\in L^2_1(M,S;\chi)$ satisfying \eqref{eqn:BVP}.
 The proof is complete.

%By \cite[Proposition 8.3]{BC}, without using the weighted Pincar\'e inequality holds,  

\end{proof}

\begin{definition}[{\cite[Definition 8.2]{BC2005}}]
 The covariant derivative $\nabla$ on $S$ over $M$ admits a weighted Poincar\'e inequality if there is a weight function $w \in L_{\text {loc }}^{1}(M)$ with $\operatorname{ess} \inf_{\Omega} w>0$ for all relatively compact $\Omega \Subset M$, such that
 \begin{equation}\label{WPI}
  \int_{M}|u|^{2} w d v_{M} \leq \int_{M}|\nabla u|^{2} d v_{M} \quad \forall u \in C_{\mathrm{c}}^{1}(M,S).
\end{equation}
\end{definition}
\begin{remark}
Note that the space $L^2_1(M,S)$ is the $\|\cdot\|_{L^2_1(M,S)}$-completion	 of $C^\infty_c(M,S)$, so the above weighted Poincar\'e inequality also holds for any $u\in L^2_1(M,S)$.
\end{remark}

\begin{proposition}\label{pro:WPI}
	The connection $\wt{\n}$ on $S$ over $M$ admits a weighted Poincar\'e inequality. 
\end{proposition}
\begin{proof}
From \cite[Proposition 8.3 (2)]{BC2005}, we only need to prove that there are no nontrivial globally parallel sections. If $\wt{\n}u=0$ on $M$, then $\n^-u_1=0$ and $\n^+u_2=0$ on $M$, which follows that $u\equiv 0$ since a non-trivial section varies exponentially along geodesic segment. 
\end{proof}

\begin{proposition}\label{BVP1}
Let $(M,g)$ be a complete connected asymptotically hyperbolic spin manifold with compact boundary and let $\psi\in C^\infty_c(M,\mb{R})$ be such that $\eta_\psi\geq 0$. Write $\theta_\psi=\theta_+-\theta_-$ with $\theta_\pm\geq 0$. Suppose that $\mr{supp}(\theta_-)\subseteq M_0$, which contains at least one asymptotically hyperbolic end of $M$, and $\theta_-< w$ a.e., where $w$ is a weight of Poincar\'e inequality of $\wt{\n}$. Then for any $u\in C^\infty(M,S)$ such that $\mc{B}_\psi u\in L^2(M,S)$, there exists a unique $v\in L^2_1(M,S;\chi)$ solving the boundary value problem
\begin{align*}
\begin{split}
   \begin{cases}
 	\mc{B}_\psi v=-\mc{B}_\psi u&\text{in } M, \\
 	\chi(v|_{\p M})=v|_{\p M} &\text{on } \p M.
 \end{cases}
 \end{split}
\end{align*}
\end{proposition}
\begin{proof}
	The proof is similar to the proof of Proposition \ref{BVP}.
	For any $u\in L^2_1(M,S;\chi)$, one has
	\begin{align*}
\begin{split}
  \|\mc{B}_\psi u\|^2_{L^2(M,S)}&\geq  \|\widetilde{\n}u\|^2_{L^2(M,S)}+\int_M {\theta}_\psi |u|^2 dV+\int_{\p M}{\eta}_\psi|u|^2 dS\\
  &\geq\int_M (w-\theta_-)|u|^2 dV,
 \end{split}
\end{align*}
which follows that $\ker\mc{B}_\psi=\{0\}$. On the other hand, $\mc{B}_\psi$ is coercive at infinity by the same proof as Proposition \ref{BVP}, and the proof is complete.
\end{proof}

\bibliographystyle{alpha}
\bibliography{AH-manifolds}
\end{document}